\documentclass[a4paper,10pt]{amsart}

\usepackage[USenglish]{babel}

\usepackage{amsmath}
\usepackage{amssymb}
\usepackage{amsthm}
\usepackage{verbatim}
\usepackage{stackrel}

\usepackage[shortlabels]{enumitem}
\usepackage{comment}	
\usepackage{hyperref}
\hypersetup{
    colorlinks=true,
    linkcolor=blue,
    citecolor=red,
    filecolor=magenta,      
    urlcolor=cyan,
    pdftitle={Spectrum and convergence of eventually positive operator semigroups},
    bookmarks=true,
    linktocpage=true,
}

\usepackage[utf8]{inputenc}

\usepackage{amssymb}
\usepackage{amsmath}
\usepackage{amsthm}
\usepackage{bbm}
\usepackage{verbatim, stackrel}

\usepackage[shortlabels]{enumitem}
\usepackage{tikz}
\usepackage{marginnote}
\usepackage{color}

\newcommand{\bbN}{\mathbb{N}}

\newcommand{\bbR}{\mathbb{R}}

\newcommand{\bbT}{\mathbb{T}}

\newcommand{\bbZ}{\mathbb{Z}}

\newcommand{\calK}{\mathcal{K}}
\newcommand{\calL}{\mathcal{L}}

\newcommand{\calR}{\mathcal{R}}
\newcommand{\calS}{\mathcal{S}}

\newcommand{\calU}{\mathcal{U}}

\DeclareMathOperator{\re}{Re} 
\newcommand{\argument}{\mathord{\,\cdot\,}} 
\newcommand{\dx}{\;\mathrm{d}} 
\DeclareMathOperator{\Fix}{Fix} 
\newcommand{\norm}[1]{\left\lVert #1 \right\rVert} 
\newcommand{\modulus}[1]{\left\lvert #1 \right\rvert} 
\DeclareMathOperator{\Ima}{Rg} 
\newcommand\restrict[1]{\raisebox{-.5ex}{$|$}_{#1}} 

\newcommand{\spec}{\sigma} 
\newcommand{\Res}{\mathcal{R}} 
\newcommand{\appSpec}{\spec_{\operatorname{app}}} 
\newcommand{\pntSpec}{\spec_{\operatorname{pnt}}} 
\newcommand{\perSpec}{\spec_{\operatorname{per}}} 
\newcommand{\spr}{r} 
\newcommand{\spb}{s} 
\newcommand{\gbd}{\omega} 

\newcommand{\Implies}[2]{``\ref{#1} $\Rightarrow$ \ref{#2}'':}

\theoremstyle{definition}
\newtheorem{definition}{Definition}[section]
\newtheorem{remark}[definition]{Remark}
\newtheorem{remarks}[definition]{Remarks}
\newtheorem*{remark*}{Remark}
\newtheorem*{remarks*}{Remarks}

\theoremstyle{plain}
\newtheorem{proposition}[definition]{Proposition}
\newtheorem{lemma}[definition]{Lemma}
\newtheorem{theorem}[definition]{Theorem}
\newtheorem{corollary}[definition]{Corollary}

\numberwithin{equation}{section} 

\begin{document}

\title[Spectrum and convergence]{Spectrum and convergence of eventually positive operator semigroups}
\author{Sahiba Arora}
\address{Sahiba Arora, Technische Universität Dresden, Institut für Analysis, Fakultät für Mathematik , 01062 Dresden, Germany}
\email{sahiba.arora@mailbox.tu-dresden.de}
\author{Jochen Gl\"uck}
\address{Jochen Gl\"uck, Universität Passau, Fakultät für Informatik und Mathematik, 94032 Passau, Germany}
\email{jochen.glueck@uni-passau.de}
\date{\today}

\begin{abstract}
	An intriguing feature of positive $C_0$-semigroups on function spaces (or more generally on Banach lattices) is that their long-time behaviour is much easier to describe than it is for general semigroups. In particular, the convergence of semigroup operators (strongly or in the operator norm) as time tends to infinity can be characterized by a set of simple spectral and compactness conditions.
	
	In the present paper, we show that similar theorems remain true for the larger class of (uniformly) eventually positive semigroups -- which recently arose in the study of various concrete differential equations.
	
	A major step in one of our characterizations is to show a version of the famous Niiro--Sawashima theorem for eventually positive operators. Several proofs for positive operators and semigroups do not work in our setting any longer, necessitating different arguments and giving our approach a distinct flavour.
\end{abstract}

\maketitle

\section{Introduction}
\label{section:introduction}

\subsection*{Eventual positivity}

While positive $C_0$-semigroups are, today, a well-established tool for the analysis of linear evolution equations that respect an order structure on the underlying space, it was only recently that a more subtle behaviour was observed for a number of infinite-dimensional differential equations. For instance, consider the bi-harmonic heat equation on $\bbR^d$ \cite{FerreroGazzolaGrunau2008, GazzolaGrunau2008}, or the parabolic equation associated with the Dirichlet-to-Neumann operator on the unit circle \cite{Daners2014}; for positive initial values, the solutions to these equations first switch sign but eventually become, and stay, positive.
	
After these observations, this kind of \emph{eventually positive} behaviour became the subject of a general theory of \emph{eventually positive semigroups} that started with the papers \cite{DanersGlueckKennedy2016a, DanersGlueckKennedy2016b}. Based on the theory developed so far, numerous further examples of concrete evolution equations have been shown to exhibit eventually positive solutions. This range of examples includes
\begin{enumerate}[(i)]
	\item 
	biharmonic equations with various types of zero boundary conditions on bounded domains (\cite[Section~6.4]{DanersGlueckKennedy2016a}, \cite[Theorem~6.1 and Proposition~6.6]{DanersGlueckKennedy2016b}, and \cite[Theorem~4.4]{DanersGlueck2018b}),
	
	\item 
	heat equations with non-local boundary conditions (\cite[Theorems~6.10, 6.11 and~6.13]{DanersGlueckKennedy2016b}, and \cite[Theorems~4.2 and~4.3]{DanersGlueck2018b}) as well as their self-adjoint bounded perturbations (\cite[Example~4.10]{DanersGlueck2018a}),
	
	\item 
	various delay differential equations (\cite[Section~6.5]{DanersGlueckKennedy2016a}, \cite[Section~11.6]{Glueck2016}, and \cite[Theorem~4.6]{DanersGlueck2018b}),
	 
	\item
	bi-Laplacians on graphs (\cite[Section~6]{GregorioMugnolo2020}), and
	
	\item
	the bi-Laplace operator with Wentzell boundary conditions (\cite[Section~7]{DenkKunzePloss2020}).
\end{enumerate}

Likewise, the closely related properties \emph{asymptotic positivity}, \emph{local eventual positivity}, and \emph{eventual domination} have been shown to occur in many concrete evolution equations; for examples of those we refer the reader to \cite[Chapter~8]{Glueck2016}, \cite[Section~5]{Arora2021}, \cite[Section~3]{AddonaGregorioRhandiTacelli2021},  and \cite[Section~4]{GlueckMugnolo2021}.
On finite-dimensional spaces, a similar theory of eventual positivity had already emerged several years earlier, as can, for instance, be seen in the paper \cite{NoutsosTsatsomeros2008}.

Despite this wealth of concrete examples, our current understanding of the theoretical features of eventually positive semigroups lags considerably behind that of the classical positive case. This makes a thorough investigation of various theoretical properties of eventually positive semigroups a worthwhile endeavour.
The present paper is part of this investigation, and it focuses on \emph{convergence to equilibrium} for such semigroups -- a topic which has been intensively studied for positive semigroups, but which currently lacks similar thorough coverage for the eventually positive case.

\subsection*{Spectrum and long-time behaviour}

One of the fundamental properties of positive (or, as they are often called within matrix analysis, \emph{non-negative}) matrices and matrix semigroups is their special spectral behaviour. This behaviour is studied and described -- in the finite-dimensional case as well as for operators and operator semigroups in infinite dimensions -- in what is today known as \emph{Perron--Frobenius theory} or \emph{Kre\u{\i}n--Rutman theory}. A striking consequence of the results of this theory is that the long-time behaviour of positive $C_0$-semigroups is subject to certain restrictions -- an observation that greatly facilitates the proof of convergence to equilibrium (as $t \to \infty$) of many positive semigroups. For details, we refer, for instance, to the book chapters \cite[Chapters~B-IV and~C-IV]{Nagel1986} and \cite[Chapter~14]{BatkaiKramarRhandi2017} and, specifically to the results in \cite{Thieme1998}.

\subsection*{Contributions of this article}

The purpose of this article is to show that similar convergence results, as for positive semigroups remain true in the case of eventual positivity. While some of our techniques will be familiar to the experts in the field of positive semigroups, we have to differ from classical arguments on several occasions since our semigroups are positive only for large times, which renders some of the known methods unsuitable. This particularly concerns the point that it is unknown whether the generator of a (uniformly) eventually positive semigroup has cyclic peripheral spectrum under the same general conditions as in the positive case; we do not solve this question here, and thus have to circumvent this problem by using different arguments than usual to obtain convergence as $t \to \infty$.

Specifically, we prove three types of main results:

\begin{enumerate}[(i)]
	\item In Sections~\ref{section:strong-convergence-i} and~\ref{section:strong-convergence-ii} we give criteria for the \emph{strong convergence} of an eventually positive $C_0$-semigroup as $t \to \infty$ (Corollary~\ref{cor:strong-convergence-i}, Theorem~\ref{thm:strong-convergence-ii}, and Corollary~\ref{cor:strong-convergence-ii-reflexive}).
	
	\item Section~\ref{section:niiro-sawashima} is an intermezzo where we study single operators with eventually positive powers rather than eventually positive semigroups. For such operators, we prove a Niiro--Sawashima type result which allows us to transfer information about the spectral radius of an eventually positive operator to other peripheral spectral values (Theorem~\ref{thm:niiro-sawashima}). This is another instance where our proof has to differ considerably from the standard proofs for the positive case.
	
	Besides being interesting in its own right, our Niiro--Sawashima type theorem is also a very crucial component for the study of our third topic:
	
	\item In Section~\ref{section:uniform-convergence} we characterize \emph{operator norm convergence} of eventually positive $C_0$-se\-mi\-groups as $t \to \infty$ (Theorems~\ref{thm:uniform-convergence} and \ref{thm:uniform-convergence-finite-residuum}).
\end{enumerate}

It is important to point out that our objectives in this paper are theoretical in nature.
Under much more restrictive and technical assumptions than presented here, the convergence of eventually positive semigroups is known and is even an essential ingredient in the characterization of eventual positivity;
see, for instance, \cite[Theorem~5.2]{DanersGlueckKennedy2016b} for strong convergence, and combine it with \cite[Corollary~2.5]{DanersGlueck2017} for operator norm convergence.
In particular, for most concrete semigroups that are today known to be eventually positive, convergence as $t \to \infty$ is well-understood.

So the purpose of our present paper is not to better understand those concrete examples, but: 
(i) to prove convergence results under weaker and less technical assumptions, thus exploring the current limits of the theory;
(ii) to bring the state of the art in eventual positivity closer to what is already known for positive semigroups, thereby highlighting the methodological differences between both situations;
and (iii) to provide necessary conditions for semigroups to be eventually positive, which narrows down the future path towards examples of eventually positive semigroups that do not fall within the framework studied in \cite{DanersGlueckKennedy2016b}.

\subsection*{A bit of terminology}

Throughout the paper, we assume the reader to be familiar with the theory of real and complex Banach lattices; standard references for this topic include the monographs \cite{Schaefer1974, Meyer-Nieberg1991} as well as the double volume \cite{LuxemburgZaanen1971, Zaanen1983}. As a standard example of a complex Banach lattice, it is often helpful to keep a complex-valued $L^p$-space (over any measure space) in mind. So, let $E$ be a complex Banach lattice. We denote the space of bounded linear operators on $E$ by $\calL(E)$. The range of an operator $T \in \calL(E)$ will be denoted by $\Ima T$.

A $C_0$-semigroup on $E$ with generator $A$ -- for which we use the convenient notation $(e^{tA})_{t \in [0,\infty)}$ -- is called \emph{individually eventually positive} if, for each $0 \le f \in E$, there exists a time $t_0 \in [0,\infty)$ such that $e^{tA}f \ge 0$ for all $t \ge t_0$; the semigroup is called \emph{uniformly eventually positive} if $t_0$ can be chosen to be independent of $f$. For two examples which demonstrate that uniform eventual positivity is indeed stronger than individual eventual positivity, we refer to \cite[Examples~5.7 and~5.8]{DanersGlueckKennedy2016a}.

Clearly, a $C_0$-semigroup $(e^{tA})_{t \in [0,\infty)}$ is uniformly eventually positive if and only if there exists a time $t_0 \in [0,\infty)$ such that $e^{tA}$ is a positive operator for each time $t \ge t_0$; here, an operator $T: E \to E$ is called \emph{positive} if $Tf \ge 0$ for each $0 \le f \in E$. There are various other types of eventual positivity, which are for instance discussed in \cite[Definition~5.1]{DanersGlueckKennedy2016a} and \cite[Definitions~5.1 and~8.1]{DanersGlueckKennedy2016b}
-- but our main focus is on uniformly eventually positive semigroups. 

Throughout the paper, it is useful to keep in mind that the spectral bound $\spb(A)$ of the generator $A$ -- which is defined as
\[
	\spb(A):=\sup\{\re \lambda: \lambda\in \spb(A)\} \in [-\infty,\infty]
\]
-- of an individually eventually positive semigroup is either $-\infty$ or an element of the spectrum of $\spec(A)$; this was proved in \cite[Theorem~7.6]{DanersGlueckKennedy2016a}.

A $C_0$-semigroup $(e^{tA})_{t \ge 0}$ on $E$ is called \emph{real} if each operator $e^{tA}$ is \emph{real}, by which we mean that it leaves the real part of the complex Banach lattice $E$ invariant. Further notation and terminology are introduced as they are needed.

\section{Strong convergence of eventually positive semigroups I}
\label{section:strong-convergence-i}

In this section, we give a first criterion for an eventually positive semigroup to converge strongly as $t \to \infty$. In general, it is known that a $C_0$-semigroup $(e^{tA})_{t \in [0,\infty)}$ on a (say, complex) Banach space $E$ converges strongly as $t \to \infty$ if and only if all orbits are relatively compact in $E$ and the point spectrum $\pntSpec(A)$  of $A$ intersects the imaginary axis at most in $0$. This follows, for instance, from the Jacobs--de Leeuw--Glicksberg decomposition of operator semigroups; see e.g.\ \cite[Theorem~V.2.14]{EngelNagel2000}.

If $E$ is a Banach lattice and our semigroup is eventually positive, the aforementioned spectral condition can be formally weakened: it suffices to assume that the intersection of the point spectrum $\pntSpec(A)$ with the imaginary axis is bounded. This is a consequence of the following theorem which, in the language of Perron--Frobenius theory, can be called a \emph{cyclicity result}.

\begin{theorem}[Cyclicity of the peripheral point spectrum]
	\label{thm:cyclic-point-spectrum}
	
	Let $(e^{tA})_{t \in [0,\infty)}$ be an individually eventually positive $C_0$-semigroup on a complex Banach lattice $E$. Assume that, for each $f \in E$, the orbit $\{e^{tA}f: \, t \in [0,\infty)\}$ is relatively compact with respect to the weak topology on $E$.
	
	If $i\beta \in i \bbR$ is an eigenvalue of $A$, then $in\beta$ is also an eigenvalue of $A$ for each $n \in \bbZ$.
\end{theorem}

This was proved, even under slightly weaker assumptions, in the second named author's PhD thesis \cite[Theorem~6.3.2]{Glueck2016}. For the convenience of the reader, we include the main arguments of the proof here. An analogous result for the single operator case was proved by similar techniques in \cite[Theorem~8.1]{Glueck2017}.

\begin{proof}[Proof of Theorem~\ref{thm:cyclic-point-spectrum}]
	We employ the so-called Jacobs--de Leeuw--Glicksberg (JdLG) decomposition on operator semigroups as it is, for instance, presented in \cite[Sections~V.2(a) and~(b)]{EngelNagel2000}:
	
	Let $\calS$ denote the closure of the set $\{e^{sA}: \, s \in [0,\infty)\}$ with respect to the weak operator topology; then $\calS$ is, with respect to operator multiplication and the weak operator topology, a compact semi-topological semigroup. The set
	\begin{align*}
		\calK := \bigcap_{T \in \calS} T \cdot \calS
	\end{align*}
	is the unique minimal ideal in the semigroup $\calS$ and it is, at the same time, a compact topological group (\cite[Theorem~V.2.3]{EngelNagel2000}).
	
	Next, we note that each operator in $\calK$ is positive. To see this, we denote the weak operator closure of each set $\calR \subseteq \calL(E)$ by $\overline{\calR}^w$. Then we obtain the inclusion
	\begin{align*}
		\calK = \bigcap_{T \in \calS} \overline{\{ T e^{sA}: \, s \in [0,\infty) \}}^w \subseteq \bigcap_{t \in [0,\infty)} \overline{\{ e^{(t+s)A}: \, s \in [0,\infty) \}}^w,
	\end{align*}
	where the first equality follows from the compactness of $\calS$ and from the separate weak continuity of operator multiplication. The set on the right is precisely the set of all accumulation points of the net $(e^{tA})_{t \in [0,\infty)}$ with respect to the weak operator topology, so each operator in $\calK$ is such an accumulation point. The individual eventual positivity of our semigroup thus implies that each operator in $\calK$ is positive.
	
	We now proceed to apply the standard JdLG technique: let $Q$ denote the neutral element in the group $\calK$. Then $Q$ is a projection that commutes with our semigroup, and the range $\Ima Q$ of $Q$ is precisely the closed linear span of the eigenvectors of $A$ corresponding to the eigenvalues on the imaginary axis (\cite[Theorem~V.2.8(ii)]{EngelNagel2000}); moreover, the restricted semigroup $\left(e^{tA}\restrict{\Ima Q}\right)_{t \in [0,\infty)}$ extends to a bounded $C_0$-group on $\Ima Q$, whose operators at negative times are also restrictions of operator from $\calK$ to the invariant subspace $\Ima Q$ \cite[p.\,315]{EngelNagel2000}.
	
	Finally, we note that $\Ima Q$ is itself a Banach lattice (with respect to the cone induced by $E$ and an equivalent norm) since the projection $Q$ is positive; this is explained in \cite[Proposition~III.11.5]{Schaefer1974}. As all operators in $\calK$ are positive, the bounded $C_0$-group $\left(e^{tA}\restrict{\Ima Q}\right)_{t \in [0,\infty)}$ is positive. It now follows from \cite[Theorem~C-III-4.2]{Nagel1986} that the point spectrum of this group's generator satisfies the property that we claim in the theorem. Since the eigenvalues of this generator coincide with the eigenvalues of $A$ that are located on $i\bbR$, our theorem follows.
\end{proof}

We remark that a semigroup whose orbits are relatively compact with respect to the weak topology is sometimes called \emph{weakly almost periodic}.

Theorem~\ref{thm:cyclic-point-spectrum} is a typical example of a so-called \emph{cyclicity result}: If the spectral  bound $\spb(A)$ equals $0$ (note that the spectral bound cannot be strictly positive since the semigroup is bounded, and in the case $\spb(A) < 0$, the assertion of the theorem is trivial), the set ${\pntSpec(A) \cap i \bbR}$ is the so-called \emph{peripheral point spectrum} of $A$, and the conclusion of the theorem can be phrased by saying that this peripheral point spectrum is \emph{cyclic}. For positive semigroups, cyclicity results for the peripheral point spectrum have been known for a long time; see for instance \cite[Corollary~C-III-4.3]{Nagel1986}.

With the aid of Theorem~\ref{thm:cyclic-point-spectrum}, it can be shown that the semigroup generated by a third order operator with periodic boundary conditions is not individually eventually positive. In fact, the same can be said about higher odd order operators; see \cite[Proposition~6.10]{AroraGlueck2021}.

As a consequence of the previous cyclicity theorem, we obtain the following characterization of strong convergence that we mentioned at the beginning of this section.

\begin{corollary}[Strong convergence I]
	\label{cor:strong-convergence-i}
	
	Let $(e^{tA})_{t \in [0,\infty)}$ be an individually eventually positive $C_0$-semigroup on a complex Banach lattice $E$. Then the following assertions are equivalent:
	\begin{enumerate}
		\item The limit $\lim_{t \to \infty} e^{tA}f$ exists for every $f \in E$.
		
		\item The orbit $\{e^{tA}f: \, t \in [0,\infty)\}$ is relatively compact in $E$ for every $f \in E$. Moreover, the intersection of the point spectrum $\pntSpec(A)$ with the imaginary axis $i\bbR$ is bounded.
	\end{enumerate}
\end{corollary}
\begin{proof}
	For semigroups with relatively compact orbits, it follows from Theorem~\ref{thm:cyclic-point-spectrum} that the intersection of $\pntSpec(A)$ with the imaginary axis is bounded if and only if it is contained in $\{0\}$. The assertions thus follow from the Jacobs--de Leeuw--Glicksberg decomposition of semigroups with respect to the strong operator topology as described in \cite[Theorem~V.2.14]{EngelNagel2000}.
\end{proof}

We point out that the boundedness of the intersection $\pntSpec(A) \cap i \bbR$ in Corollary~\ref{cor:strong-convergence-i}(ii) is automatically satisfied if, for example, the semigroup is analytic or, more generally, eventually norm continuous -- because in these cases, even the intersection of the entire spectrum with $i\bbR$ is bounded. On the other hand, even more can be said for semigroups that satisfy such an additional regularity assumption. This is the content of the next section.

\section{Strong convergence of eventually positive semigroups II}
\label{section:strong-convergence-ii}

An interesting point about Corollary~\ref{cor:strong-convergence-i} is that we only need information about the eigenvalues of $A$ rather than about all spectral values to obtain strong convergence. On the downside, though, we need a priori information that all orbits are relatively norm compact, which is a rather strong requirement.

In this section, we show that we can dispense with the compactness condition in many cases. Indeed, a typical situation where we know that the intersection of $\pntSpec(A)$ with the imaginary axis is bounded is in the case where the semigroup is eventually norm continuous. But in this situation, we even know that the spectrum $\sigma(A)$ is bounded on the imaginary axis; for positive semigroups, this is extremely helpful since for such semigroups, under mild technical assumptions, a cyclicity result such as in Theorem~\ref{thm:cyclic-point-spectrum} holds for the entire \emph{peripheral spectrum} \cite[Theorem~C-III-2.10]{Nagel1986}. Consequently, eventual norm continuity together with positivity (plus appropriate technical assumptions such as, say, boundedness) implies that $\spec(A) \cap i \bbR$ contains at most the number $0$. With this knowledge, one can then apply Tauberian theorems (for instance the ABLV theorem as in \cite[Theorem~V.2.21]{EngelNagel2000}) to derive strong convergence.

The problem with this approach is that a cyclicity result for the entire peripheral spectrum as in \cite[Theorem~C-III-2.10]{Nagel1986} is currently not known for eventually positive semigroups, and the techniques used in the positive case do not seem to be easily adaptable to the eventual positive case. On the other hand, at least for single operators, a cyclicity result has been proved in \cite[Theorem~7.1]{Glueck2017}. Thus, our approach here is to use a spectral mapping theorem in order to study the spectrum of the single semigroup operators $e^{tA}$ rather than the spectrum of the generator $A$ and to apply the known cyclicity result to $e^{tA}$ (for an appropriate choice of $t$). Doing so enables us to show the following Theorem~\ref{thm:strong-convergence-ii}.

As mentioned before, the theorem is mainly concerned with the case of eventually norm continuous semigroups. Since it does not pose any additional difficulties, we state and prove the theorem for a slightly more general class of semigroups, namely for semigroups that are \emph{norm continuous at infinity}. This notation, which was introduced by Mart\'{\i}nez and Maz\'on in \cite[Definition~1.1]{MartinezMazon1996}, is a bit more general than eventual norm continuity, but it is sufficiently strong to guarantee similar spectral results and hence similar convergence properties as in the eventually norm continuous case; in \cite[Theorem~3.3 and Corollary~3.4]{MartinezMazon1996}, this is demonstrated for strong convergence of positive semigroups. We 
show that these results remain true for eventually positive semigroups as well; however, the proofs cannot be directly adapted and we need to follow a different approach. Recall that a $C_0$-semigroup $(e^{tA})_{t \in [0,\infty)}$ with growth bound $0$ on a Banach space $E$ is called \emph{norm continuous at infinity} if 
\begin{align*}
	\limsup_{s \to 0} \norm{e^{tA} - e^{(t+s)A}} \to 0 \quad \text{as } t \to \infty;
\end{align*}
if the growth bound $\gbd(A)$ is not $0$ but different from $-\infty$, then the semigroup is called \emph{norm continuous at infinity} if the same is true for the rescaled semigroup $(e^{t(A-\gbd(A)I)})_{t \in [0,\infty)}$. From now on, whenever we say that a semigroup is norm continuous at infinity, we tacitly include the assumption $\gbd(A) > -\infty$ in this wording. Note that this implies that we also have $\spb(A) > -\infty$ since the growth bound $\gbd(A)$ coincides with the spectral bound $\spb(A)$ according to \cite[Corollary~1.4(i)]{MartinezMazon1996}.

\begin{theorem}[Strong convergence II]
	\label{thm:strong-convergence-ii}

	Let $(e^{tA})_{t \in [0,\infty)}$ be a uniformly eventually positive $C_0$-semigroup on a complex Banach lattice $E$. If the semigroup is norm continuous at infinity and bounded, then the following are equivalent:
	\begin{enumerate}[ref=(\roman*)]
		\item\label{thm:strong-convergence-ii:item:convergence} The limit $\lim_{t \to \infty} e^{tA}f$ exists for every $f \in E$.
		\item\label{thm:strong-convergence-ii:item:mean-ergodic} The semigroup is mean ergodic.
	\end{enumerate} 
\end{theorem}

Here, \emph{mean ergodic} means that, for each $f \in E$, the \emph{Cesàro means} $\frac{1}{t}\int_0^t e^{sA}f \dx s$ converge strongly as $t \downarrow 0$. Since every bounded $C_0$-semigroup on a reflexive Banach space is automatically mean ergodic, the theorem yields the following corollary:

\begin{corollary}
	\label{cor:strong-convergence-ii-reflexive}

	Let $(e^{tA})_{t \in [0,\infty)}$ be a uniformly eventually positive $C_0$-semigroup on a complex and reflexive Banach lattice $E$. If the semigroup is norm continuous at infinity and bounded, then the limit $\lim_{t \to \infty} e^{tA}f$ exists for every $f \in E$.
\end{corollary}

As explained at the beginning of the section, we are now going to prove Theorem~\ref{thm:strong-convergence-ii} by using spectral properties of the operators $e^{tA}$. We outsource the essence of the argument to the following lemma. For positive semigroups, this was proved in \cite[Proposition~3.2]{MartinezMazon1996}. If $(e^{tA})_{t \in [0,\infty)}$ is a $C_0$-semigroup and the spectral bound of $A$ is not $-\infty$, then we call the set
\begin{align*}
	\perSpec(A) := \spec(A) \cap (\spb(A) + i\bbR),
\end{align*}
the \emph{peripheral spectrum} of $A$.

\begin{lemma}
	\label{lem:nci-dominant-bounded}
	On a Banach lattice $E$, let $(e^{tA})_{t \in [0,\infty)}$ be a uniformly eventually positive semigroup that is norm continuous at infinity. If $(e^{tA})_{t \in [0,\infty)}$ is bounded and $\spb(A) = 0$, then $\perSpec(A) = \{0\}$.
\end{lemma}

We will show a second result that is close in spirit to Lemma~\ref{lem:nci-dominant-bounded}, but has different technical assumptions, in Lemma~\ref{lem:nci-dominant-pole}.
	
\begin{proof}[Proof of Lemma~\ref{lem:nci-dominant-bounded}]
	As the semigroup is bounded and the spectral bound is $0$, it follows that the growth bound is $0$, too. Moreover, the norm continuity at infinity implies that $\perSpec(A)$ is bounded; this was proved in \cite[Theorem~1.9]{MartinezMazon1996}. We can thus find a number $\alpha>0$ such that $\perSpec(A) \subseteq [-i\alpha,i\alpha]$. 
	
	Again, due to the norm continuity at infinity, a spectral mapping theorem holds for the peripheral spectrum \cite[Theorem~1.2]{MartinezMazon1996}, so we have
	\begin{equation*}
		\spec(e^{tA}) \cap \bbT \subseteq \{e^{it\gamma}:\gamma \in [-\alpha,\alpha]\}
	\end{equation*}
	for all $t \geq 0$; here $\bbT$ denotes the complex unit circle. For all sufficiently small $t$, say $t \le t_0$, we conclude that the spectrum of $e^{tA}$ does not intersect the left half of the unit circle $\bbT$. Fix such a time $t \in (0,t_0]$.
	
	Now we use a cyclicity result for single operators: By assumption, $e^{tA}$ is power bounded and all its powers with sufficiently large exponent are positive. This implies that, whenever $\lambda$ is a spectral value of $e^{tA}$ of modulus $\modulus{\lambda} = 1$, then all powers $\lambda^n$ (for $n \in \bbZ$) are spectral values of $e^{tA}$, too; see \cite[Theorem~7.1]{Glueck2017} for a proof.
	
	Given that the spectrum of $e^{tA}$ does not intersect the left half of $\bbT$, this implies that 
	\begin{equation*}
		\spec(e^{tA}) \cap \bbT = \{1\}.
	\end{equation*}
	Since this is true for all times $t \in (0,t_0]$ it follows (for instance, by choosing a time ${t < \pi/\alpha}$) from the spectral inclusion theorem for $C_0$-semigroups \cite[Theorem~IV.3.6]{EngelNagel2000} that $\perSpec(A)\subseteq \{0\}$.
	
	On the other hand, it was proved in \cite[Theorem~7.6]{DanersGlueckKennedy2016a} that individual eventual positivity of the semigroup implies $\spb(A) \in \spec(A)$; so we actually have ${\perSpec(A) = \{0\}}$.
\end{proof}

Given this lemma, it is now easy to derive Theorem~\ref{thm:strong-convergence-ii} from a Tauberian theorem:

\begin{proof}[Proof of Theorem~\ref{thm:strong-convergence-ii}]
	\Implies{thm:strong-convergence-ii:item:convergence}{thm:strong-convergence-ii:item:mean-ergodic} This implication is obvious.
	
	\Implies{thm:strong-convergence-ii:item:mean-ergodic}{thm:strong-convergence-ii:item:convergence} According to Lemma~\ref{lem:nci-dominant-bounded}, the spectrum of $A$ intersects $i\bbR$ at most in $0$. 
	If $P \in \calL(E)$ denotes the mean ergodic projection of the semigroup, then the restriction of $(e^{tA})_{t \in [0,\infty)}$ to $\ker P$ is a bounded and mean ergodic $C_0$-semigroup on $\ker P$ whose generator $B$ -- which is simply the part of $A$ in $\ker P$ -- has the following properties: the spectrum $\spec(B)$ intersects $i\bbR$ at most in $0$ and $0$ is not an eigenvalue of $B$; but since the semigroup is mean ergodic, this implies that $0$ is not an eigenvalue of the adjoint operator $B'$, either. Hence, the ABLV theorem \cite[Theorem~V.2.21]{EngelNagel2000} yields that $e^{tB} = e^{tA}\restrict{\ker P}$ converges strongly to $0$ as $t \to \infty$.
	
	On the other hand, $e^{tA}$ acts as the identity on the range of $P$, so we conclude that $e^{tA} \to P$ strongly as $t \to \infty$.
\end{proof}

We close this section with a remark on the types of eventual positivity that we assumed in this and in the previous section to deduce strong convergence.

\begin{remark}
	In Corollary~\ref{cor:strong-convergence-i}, we only needed individual eventual positivity of the semigroup. This is because the cyclicity of the peripheral point spectrum -- which is the essential ingredient for the strong convergence in Corollary~\ref{cor:strong-convergence-i} -- is proved by utilizing the Jacobs--de Leeuw--Glicksberg decomposition in Theorem~\ref{thm:cyclic-point-spectrum}; this decomposition allows us, in a sense, to consider only the behaviour of the orbits ``at infinity'' -- hence, it does not matter, when precisely each orbit becomes positive.
	
	Theorem~\ref{thm:strong-convergence-ii}, on the other hand, requires uniform eventual positivity. This is due to the cyclicity result for single operators that we employ in the proof; this cyclicity result is only shown under a uniform eventual positivity assumption in \cite[Theorem~7.1]{Glueck2017}. 
	
	We do not know whether Theorem~\ref{thm:strong-convergence-ii} remains true for individually eventually positive semigroups. As explained in \cite[Remark~7.3(b)]{Glueck2017}, it is also unknown whether the cyclicity result in \cite[Theorem~7.1]{Glueck2017} remains true for operators that are individually eventually positive.
\end{remark}

\section{A Niiro--Sawashima theorem for eventually positive operators}
\label{section:niiro-sawashima}

While strong convergence of eventually positive semigroups was the content of Sections~\ref{section:strong-convergence-i} and~\ref{section:strong-convergence-ii}, we are going to study uniform convergence -- i.e., convergence with respect to the operator norm -- in Section~\ref{section:uniform-convergence}. Whenever we speak about operator norm convergence of semigroups or powers of operators, a particular spectral condition naturally arises -- namely that the peripheral spectrum consists of poles of the resolvent only.

In this context, there exists a useful theorem which is, in its original form, due to Niiro and Sawashima (see \cite{NiiroSawashima1966a} and \cite[Main Theorem and Theorem~9.2]{NiiroSawashima1966b}); for the version that we are interested in, it is due to Lotz and Schaefer (\cite[Theorem~2]{LotzSchaefer1968}; see also \cite[Theorem~V.5.5]{Schaefer1974}) and says that, if $T$ is a positive operator and the spectral radius $\spr(T)$ is a pole of the resolvent of $T$ with finite-dimensional spectral space, then all spectral values of $T$ with modulus $\spr(T)$ are also poles.

We are going to show that the same result remains true for operators that are \emph{uniformly eventually positive}, at least if the pole order of the spectral radius is $1$. We call a bounded linear operator $T$ on a complex Banach lattice $E$ \emph{uniformly eventually positive} if there exists an integer $n_0 \in \bbN_0$ such that $T^n$ is a positive operator for each $n \ge n_0$. In \cite[Theorem~4.1]{Glueck2017}, it was proved that the spectral radius $\spr(T)$ of such an operator is always a spectral value of $T$. 

The following is the main result of this section:

\begin{theorem}[Niiro--Sawashima for uniformly eventually positive operators]
	\label{thm:niiro-sawashima}
	
	Let $E$ be a complex Banach lattice and let $T \in \calL(E)$ be uniformly eventually positive. Assume that the spectral value $\spr(T)$ is a pole of the resolvent of $T$ of pole order $1$ and that the corresponding spectral space $F \subseteq E$ is finite-dimensional. 
	
	Then every spectral value $\lambda$ of $T$ with modulus $\modulus{\lambda} = \spr(T)$ is also a first order pole of the resolvent $\Res(\argument,T)$ and the dimension of its spectral space is not larger than $\dim F$.
\end{theorem}

We call a spectral value $\lambda$ of a bounded linear operator $T$ on a complex Banach space a \emph{Riesz point} of $T$ if it is a pole of the resolvent $\Res(\argument,T)$ and if the corresponding spectral space is finite-dimensional; see the appendix for further details. Thus, the essence of Theorem~\ref{thm:niiro-sawashima} is: if the spectral radius of $T$ is a Riesz point and its pole order equals $1$, then the same is true for every other spectral value of modulus $\spr(T)$. If $\spr(T) = 1$, this means that $T$ is \emph{quasi-compact}, i.e., the essential spectral radius of $T$ is strictly less that $1$. This, together with the fact that all poles are of first order, also implies that $T$ is power bounded -- an observation which will be of significant use in Section~\ref{section:uniform-convergence}.

Before we come to the proof of Theorem~\ref{thm:niiro-sawashima}, let us make a few remarks on the relation between the theorem and its version for positive operators:

\begin{remarks}
	 (a) The version of the Niiro--Sawashima theorem for positive operators that can be found in \cite[Theorem~2]{LotzSchaefer1968} or \cite[Theorem~V.5.5]{Schaefer1974} only states that all spectral values of modulus $\spr(T)$ are poles of the resolvent; the fact that their spectral space is also finite-dimensional is not stated explicitly there, but an inspection of the proof shows that this is indeed true.
	
	(b) For positive operators, the assertion of Theorem~\ref{thm:niiro-sawashima} is also true if the pole order of $\spr(T)$ is not assumed to be $1$; see the aforementioned references \cite[Theorem~2]{LotzSchaefer1968} or \cite[Theorem~V.5.5]{Schaefer1974}. We do not know whether the same holds for uniformly eventually positive operators as well.
\end{remarks}

The proof of the Niiro--Sawashima theorem for positive operators given in \cite[Theorem~2]{LotzSchaefer1968} and \cite[Theorem~V.5.5]{Schaefer1974} relies heavily on the theory of irreducible operators, in conjunction with an ultrapower argument. For our situation, a reduction to the case of irreducible operators seems to be out of the question. Instead, we are going to use a different argument which is inspired by the following approach:

In \cite[Theorem~3.3]{Groh1984}, Groh proved a version of the Niiro--Sawashima theorem for a class of positive operators on a $C^*$-algebra; in the process, he also faced the problem that the classical arguments for positive operators on Banach lattices do not work in this setting. To solve the problem, Groh derived a description of poles of the resolvent with finite-dimensional spectral space by considering the dimension of certain eigenspaces of an ultrapower \cite[Proposition~3.2]{Groh1984}. This technique was later refined in \cite[Proposition~3.3]{Caselles1987} and \cite[Corollary~3.2]{Glueck2020}, and it often turned out to be useful in the spectral analysis of linear operators (for instance, in \cite{Lotz1986} and \cite{Martinez1993}, as well as in both aforementioned articles). We are going to use the same technique for our proof of Theorem~\ref{thm:niiro-sawashima}. For the convenience of the reader, we include a brief reminder of ultrapowers as well as a concrete statement of the relevant spectral results in the appendix.

Our strategy to prove Theorem~\ref{thm:niiro-sawashima} now is as follows: We first show a dimension estimate for the range of linear operators that are dominated by a positive projection (Lemma~\ref{lem:dominating-projection}). We then use this lemma to derive a dimension estimate for certain eigenspaces of eventually positive operators (Theorem~\ref{thm:dominating-riesz-point}). Finally, we derive Theorem~\ref{thm:niiro-sawashima} by an application of the ultrapower characterization of Riesz points given in Proposition~\ref{prop:riesz-point-via-ultrapower}. 

Let us begin with our lemma on operators that are dominated by a positive projection.

\begin{lemma}
	\label{lem:dominating-projection}
	
	Let $E$ be a complex Banach lattice and let $Q,P \in \calL(E)$ be such that
	\begin{equation}
		\label{eq:dominating-projection}
		\modulus{Qf} \leq P\modulus{f} \text{ for all } f \in E.
	\end{equation}
	If $P$ is a projection, then the fixed space $\Fix Q:= \ker(I-Q)$ and the range $\Ima P$ satisfy $\dim \Fix Q \le \dim \Ima P$.
\end{lemma}

\begin{proof}
	Note that \eqref{eq:dominating-projection} implies that $P$ is positive. We begin by proving the result for the particular case when $P$ is strictly positive, i.e., when the kernel $\ker P$ does not contain any positive non-zero elements. In this case, the modulus of each vector $f \in \Fix Q$ belongs to the range $\Ima P$. Indeed, fixing such a $f$ we have
	\[
		P\modulus{f}-\modulus{f}=P\modulus{f}-\modulus{Qf}\geq 0;
	\]
	moreover, the vector $P\modulus{f}-\modulus{f}$ is in $\ker P$ and must therefore be $0$ by strict positivity. Hence $\modulus{f}=P\modulus{f}\in \Ima P$. We have thus shown that the modulus of every vector of $\Fix Q$ is in $\Ima P$, which implies $\dim \Fix Q \leq \dim \Ima P$ by \cite[Lemma~C-III-3.11]{Nagel1986}. This verifies the assertion  in the particular case when $P$ is strictly positive.

	We now consider the general case. Let 
	\[
		I := \{f \in E : P\modulus{f}=0\}
	\]
	denote the so-called \emph{absolute kernel} of $P$. Clearly, $I$ is a closed $P$-invariant ideal and the inequality \eqref{eq:dominating-projection} implies that it is also $Q$-invariant. Let $\pi$ denote the quotient map from the Banach lattice $E$ to the quotient Banach lattice $E/I$. Since $I$ is both $P$- and $Q$-invariant, therefore $P$ and $Q$ induce operators $P_/, Q_/$ on $E/I$ respectively. 
	
	The operator $P_/$ is also a positive projection and since $\pi$ is a lattice homomorphism, we observe by using the assumption~\eqref{eq:dominating-projection} that
	\[
		\modulus{Q_/(\pi f)} = \modulus{\pi Qf} = \pi\modulus{Qf} \leq \pi P\modulus{f} = P_/(\pi\modulus{f}) = P_/ \modulus{\pi f}
	\]
	for all $f \in E$. Hence, $Q_/$ and $P_/$ also satisfy the analogous estimate of~\eqref{eq:dominating-projection} in the quotient space $E/I$. In addition, $P_/$ is strictly positive. To see this, let $\widehat{f}$ be a positive vector in $\ker P_/$ and $f$ be a positive element of $E$ such that  $\widehat{f}=\pi f$. Then 
	\[
		\pi Pf=P_/(\pi f)=P_/\widehat{f}=0,
	\]
	which means $Pf$ must be in $I$. However as $f$ and $Pf$ are positive, then, in fact, $f$ must belong to $I$. Hence, $\widehat{f}$ is the zero vector of $E/I$. Therefore $\ker P_/$ has no positive non-zero elements which shows that $P_/$ is strictly positive, as claimed. Whence by what we proved in the special case at the beginning of the proof, we can conclude that
	\[
		\dim \Fix Q_/ \leq \dim \Ima P_/.
	\]
	
	The assertion will now follow if we are able to prove the following two properties of the quotient map $\pi$:
	\begin{enumerate}[ref=(\roman*)]
		\item The restriction of $\pi$ to $\Ima P$ maps surjectively onto $\Ima P_/$. \label{lem:dom-projection:bijection1}
		\item The restriction of $\pi$ to $\Fix Q$ maps injectively into $\Fix Q_/$.\label{lem:dom-projection:bijection2}
	\end{enumerate}
	Assertion~\ref{lem:dom-projection:bijection1} follows from the fact that $\pi P = P_/ \pi$ and from the surjectivity of $\pi$. To prove \ref{lem:dom-projection:bijection2}, let $f$ be a non-zero vector belonging to $\Fix Q$. Then we have $Q_/(\pi f) = \pi(Qf) = \pi f$. In other words, $\pi$ maps $f$ to $\Fix Q_/$. Moreover, using \eqref{eq:dominating-projection}, we have
	\[
		0 \neq f = \modulus{Qf} \leq P\modulus{f},
	\]
	and so $\pi f$ is non-zero; hence the claimed injection follows. 
\end{proof}

We note in passing that, in the situation of the above proof, it can even be shown that both mappings
\begin{align*}
	\pi\restrict{\Ima P} : \Ima P \to \Ima P_/
	\qquad \text{and} \qquad 
	\pi\restrict{\Fix Q} : \Fix Q \to \Fix Q_/
\end{align*}
are  bijections. Next, we use the previous lemma to derive the following dimension estimate for eigenspaces of eventually positive operators.

\begin{theorem}
	\label{thm:dominating-riesz-point}
	
	Let $E$ be a complex Banach lattice and $T \in \calL(E)$ be uniformly eventually positive. Assume the spectral radius $\spr(T)$ is a Riesz point of $T$ and the corresponding pole order of the resolvent is $1$. Then we have
	\[
		\dim \ker(\lambda I - T)\leq \dim \ker(\spr(T) I-T)
	\]
	for every complex number $\lambda$ with modulus $\modulus{\lambda}=\spr(T)$.
\end{theorem}
\begin{proof}
	Since the result trivially holds when $\spr(T)=0$, we may assume without loss of generality that $\spr(T)=1$. Let $\lambda$ be a complex number of modulus $\modulus{\lambda} = 1$. Since $1$ is a first order pole of the resolvent $\Res(\argument,T)$, the range of the corresponding spectral projection $P$ is $\ker (I-T)$. Let $\calU$ be a free ultrafilter on $\bbN$ and $(r_n)$ be a sequence of real numbers decreasing to $1$. We consider the operator sequences
	\[
		 (Q_n) := \Big((r_n\lambda - \lambda) \Res(r_n \lambda,T)\Big) \qquad \text{and} \qquad (P_n) := \Big((r_n-1)\Res(r_n,T)\Big).
	\]
	The sequence $(P_n)$ is bounded since $1$ is a first order pole of $\Res(\argument, T)$; in fact, $(P_n)$ even converges, with respect to the operator norm, to the spectral projection $P$ of $T$ associated to the spectral value $1$.
	
	Next we observe that the sequence $(Q_n)$ is asymptotically dominated by $(P_n)$ in the following sense: There exist two sequences of operators $(R_n)$ and $(S_n)$ in $\calL(E)$ such that $\norm{R_n}, \norm{S_n} \to 0$ and such that
	\begin{align}
		\label{eq:thm:dominating-riesz-point:domination-in-proof}
		\modulus{Q_n f} \le \modulus{S_nf} + R_n\modulus{f} + P_n \modulus{f}
	\end{align}
	for each index $n$ and each vector $f \in E$. To see this, just use that there exists an exponent $k_0 \in \bbN_0$ such that $T^k$ is positive for all $k \ge k_0$, and then define
	\begin{align*}
		R_n  :=  - (r_n-1) \sum_{k=0}^{k_0-1}\frac{1}{(r_n \lambda)^{k+1}}T^k
	\end{align*}
	and \begin{align*}
		S_n :=  (r_n\lambda - \lambda) \sum_{k=0}^{k_0-1}\frac{1}{(r_n \lambda)^{k+1}}T^k
	\end{align*}
	for each $n$. Clearly, both sequences converge to $0$, and the claimed inequality~\eqref{eq:thm:dominating-riesz-point:domination-in-proof} follows readily from the Neumann series expansion of $\Res(r_n \lambda,T)$ and $\Res(r_n ,T)$ and from the fact that $\modulus{T^k f} \le T^k \modulus{f}$ for each $k \ge k_0$ due to the positivity of $T^k$.
	
	As a consequence of~\eqref{eq:thm:dominating-riesz-point:domination-in-proof}, the sequence $(Q_n)$ is bounded as well. Hence, the sequences $(Q_n)$ and $(P_n)$ induce bounded operators $\widehat{Q}$ and $\widehat{P}$ on the ultrapower $E^\calU$, respectively. Moreover since $P_n \to P$ with respect to the operator norm, we know that $\widehat{P}$ is actually the projection $P^{\calU}$. 
	
	Now we employ the inequality~\eqref{eq:thm:dominating-riesz-point:domination-in-proof} a second time: it yields that we have 
	\[
		{\modulus{\widehat{Q} g} \leq \widehat{P} \modulus{g}}
	\]
	for all $g \in E^{\calU}$. Thus, Lemma~\ref{lem:dominating-projection} is applicable and yields
	\[
		\dim \Fix \widehat{Q} \leq \dim \Ima \widehat{P}=\dim \Ima P^{\calU}=\dim \Ima P=\dim \ker(I-T),
	\]
	where the penultimate equality holds because $P$ is a finite rank projection.
	
	Finally, we show that $\dim \ker(\lambda I-T)\leq \dim \Fix \widehat{Q}$ to conclude the proof. To this end, 
	let $i : E\to E^{\calU}$ denote the canonical embedding and $f$ be a vector in $\ker(\lambda I-T)$. Then ${\Res(\mu,T)f = (\mu-\lambda)^{-1}f}$ for all $\mu$ in the resolvent set of $T$ and hence
	\[
		(r_n \lambda - \lambda) \Res(r_n \lambda,T) f= (r_n\lambda - \lambda) (r_n \lambda -\lambda)^{-1}f=f
	\]
	for all $n \in \bbN$, which implies $\widehat{Q}\, i (f)=i (f)$. We have thus shown that $i$ maps $\ker (\lambda I- T)$ to $\Fix \widehat{Q}$ and it does so injectively, as $i$ is an injection. This verifies our claim.
\end{proof}

The dimension estimate from the theorem above together with another ultrapower argument now allows us to easily derive our Niiro--Sawashima theorem:

\begin{proof}[Proof of Theorem~\ref{thm:niiro-sawashima}]
	Fix a free ultrafilter $\calU$ on $\bbN$. Since $\spr(T)$ is also a Riesz point of $T^{\calU}$ (Proposition~\ref{prop:riesz-point-via-ultrapower}), the corresponding spectral space is finite-dimensional. If $\lambda$ is any spectral value of $T$ with modulus $\modulus{\lambda} = \spr(T)$, then by Theorem~\ref{thm:dominating-riesz-point},
	\begin{align}
		\label{eq:dimension-esimtate-ultra-proof}
		\dim \ker\left(\lambda I-T^{\calU}\right) \leq \dim \ker \left(\spr(T) I-T^{\calU}\right) < \infty.
	\end{align}
	Therefore $\lambda$ is a Riesz point of $T$, again by Proposition~\ref{prop:riesz-point-via-ultrapower}.
	
	The fact that the pole order of $\Res(\argument,T)$ at $\lambda$ is $1$ follows from the inequality~\eqref{eq:thm:dominating-riesz-point:domination-in-proof} that we derived in the proof of Theorem~\ref{thm:dominating-riesz-point} which holds for each uniformly eventually positive operator.
		
	Finally, for first order poles the spectral space coincides with the eigenspace. In addition, the spectral projection of $T^{\calU}$ associated to $\lambda$ is actually the lifting of the spectral projection of $T$ associated to $\lambda$, and the same is true for the spectral projections associated to $\spr(T)$ (Proposition~\ref{prop:riesz-point-via-ultrapower}). These together with~\eqref{eq:dimension-esimtate-ultra-proof} imply the dimension estimate claimed at the end of Theorem~\ref{thm:niiro-sawashima}.
\end{proof}

We close this section with a brief remark concerning our usage of ultrapowers in the proof of Theorem~\ref{thm:niiro-sawashima}:

\begin{remark}
	In order to prove Theorem~\ref{thm:niiro-sawashima}, we first lifted the operator $T$ to an ultrapower and then we applied the dimension estimate from Theorem~\ref{thm:dominating-riesz-point} to the lifted operator $T^{\calU}$. It is interesting to note that the proof of Theorem~\ref{thm:dominating-riesz-point} itself also employs an ultrapower argument. Hence, our proof of the Niiro--Sawashima type result in Theorem~\ref{thm:niiro-sawashima} actually relies on an iterated ultrapower argument, i.e., the operator
	$
		\left(T^{\calU}\right)^{\calU}
	$
	(implicitly) occurs in the proof.
\end{remark}

\section{Uniform convergence for eventually positive semigroups}
\label{section:uniform-convergence}

As mentioned in the previous section, we are now going to characterize uniform convergence of eventually positive semigroups. In general, the rescaled version $(e^{t(A-\spb(A)I)})_{t \in [0,\infty)}$ of a $C_0$-semigroup on a (complex) Banach space $E$ converges uniformly as $t\to \infty$ if and only if it is norm continuous at infinity, the spectral bound $\spb(A)$ is a first order pole of the resolvent and the only spectral value of $A$ with the largest real part. This was proven by Thieme in \cite[Theorem~2.7]{Thieme1998}. Furthermore, he showed \cite[Theorem~3.4]{Thieme1998} that if $E$ is a Banach lattice and the semigroup is positive then the last condition -- that $\spb(A)$ is the only spectral value with the largest real part -- can be dropped. This is because when the spectral bound is a first order pole and the semigroup is positive, then the peripheral spectrum is cyclic (\cite[Proposition~C-III-2.9 and Theorem~C-III-2.10]{Nagel1986}).

Unfortunately, we do not yet know if the same can be said of eventually positive semigroups. Nevertheless, we will show in Theorem~\ref{thm:uniform-convergence} that this roadblock can be bypassed if (after appropriate rescaling) the semigroup is bounded.

Besides, uniform convergence of $C_0$-semigroups to a finite rank operator on Banach spaces has also been studied (usually called \emph{asynchronous exponential growth}). For instance, Webb proved a characterization \cite[Proposition~2.3]{Webb1987} in terms of the \emph{$\alpha$-growth bound} of the semigroup and gave some sufficient conditions in case of positive semigroups on Banach lattices \cite[Remark~2.2]{Webb1987}. Using Webb's result, Thieme provided an alternate characterization \cite[Theorem~3.3]{Thieme1998} employing the concept of norm continuity at infinity; see also \cite[Sections~4.6 and 4.7]{MagalRuan2018}. As a consequence of \cite[Theorems~3.3 and 3.4]{Thieme1998}, one obtains that the rescaling $(e^{t(A-\spb(A)I)})_{t \in [0,\infty)}$ of a positive $C_0$-semigroup on a Banach lattice converges uniformly to a finite rank operator if and only if $(e^{tA})_{t \in [0,\infty)}$ is norm continuous at infinity and $\spb(A)$ is a first order pole of the resolvent with finite-dimensional spectral space. We will show in Theorem~\ref{thm:uniform-convergence-finite-residuum} that this characterization remains true for uniformly eventually positive semigroups. Here again, the setting of eventual positivity requires different methods than its positive counterpart.

Before we state the main theorems of this section, we note that Thieme actually used the concept of \emph{essentially norm continuous} semigroups in order to characterize uniform convergence. However, it was shown by Blake in his PhD thesis \cite[Corollary~3.3.7]{Blake1999} as well as by Nagel and Poland \cite[Corollary~4.7]{NagelPoland2000} that this is equivalent to norm continuity at infinity of the semigroup.

The following theorem gives a characterization for a uniformly eventually positive semigroup $(e^{tA})_{t \in [0,\infty)}$ to be \emph{uniformly exponentially balancing}; here, we call a $C_0$-semigroup $(e^{tA})_{t \in [0,\infty)}$ \emph{uniformly exponentially balancing} if $\spb(A) > -\infty$ and the rescaled semigroup $(e^{t(A-\spb(A)I)})_{t \in [0,\infty)}$ converges in the operator norm (cf.\ \cite[Definition~2.1 and Proposition~2.3]{Thieme1998}); note that the limit operator is automatically non-zero in this case.

\begin{theorem}
	\label{thm:uniform-convergence}
	
	Let $(e^{tA})_{t \in [0,\infty)}$ be a uniformly eventually positive $C_0$-semigroup on a complex Banach lattice $E$ and assume that $\spb(A) > -\infty$. The semigroup $(e^{tA})_{t \in [0,\infty)}$ is uniformly exponentially balancing if and only if the following three conditions are satisfied:
	\begin{enumerate}[ref=(\roman*)]
		\item\label{thm:uniform-convergence:item:norm-continuous-at-infinity} The semigroup $(e^{tA})_{t \in [0,\infty)}$ is norm continuous at infinity.
		
		\item\label{thm:uniform-convergence:item:bounded} The rescaled semigroup $(e^{t(A-\spb(A)I)})_{t \in [0,\infty)}$ is bounded.
		
		\item\label{thm:uniform-convergence:item:pole} The spectral bound $\spb(A)$ is a pole of the resolvent of $A$.
	\end{enumerate}
\end{theorem}

\begin{proof}
	The necessity of \ref{thm:uniform-convergence:item:norm-continuous-at-infinity}-\ref{thm:uniform-convergence:item:pole} is proved in \cite[Theorems~2.4 and 2.7]{Thieme1998}. For the converse, assume that \ref{thm:uniform-convergence:item:norm-continuous-at-infinity}-\ref{thm:uniform-convergence:item:pole} hold and that without loss of generality $\spb(A)=0$.  Then by Lemma~\ref{lem:nci-dominant-bounded}, the spectrum of $A$ intersects $i\bbR$ at the point $0$. Moreover, since the semigroup is bounded, $0$ is in fact a first order pole of the resolvent of $A$. Therefore $(e^{tA})_{t \in [0,\infty)}$ satisfies all the conditions sufficient for uniform convergence as $t \to \infty$ (\cite[Theorem~2.7]{Thieme1998}).
\end{proof}

As mentioned before, the characterization given in \cite[Theorem~3.3]{Thieme1998} for a semigroup to be uniformly exponentially balancing to a \emph{finite rank operator} can be improved for uniformly eventually positive semigroups. This is the content of the following theorem:

\begin{theorem}
	\label{thm:uniform-convergence-finite-residuum}
	
	Let $(e^{tA})_{t \in [0,\infty)}$ be a uniformly eventually positive $C_0$-semigroup on a complex Banach lattice $E$ and assume that $\spb(A) > -\infty$. The semigroup $(e^{tA})_{t \in [0,\infty)}$ is uniformly exponentially balancing and $\lim_{t \to \infty} e^{t(A-\spb(A)I)}$ has finite rank if and only if the following two conditions hold:
	\begin{enumerate}[ref=(\roman*)]
		\item\label{thm:uniform-convergence-finite-residuum:item:norm-continuous-at-infinity} The semigroup $(e^{tA})_{t \in [0,\infty)}$ is norm continuous at infinity.
		
		\item\label{thm:uniform-convergence-finite-residuum:item:simple-pole} The spectral bound $\spb(A)$ is a first order pole of the resolvent of $A$ and the corresponding spectral space is finite-dimensional.
	\end{enumerate}
\end{theorem}

We remark that the notion \emph{Riesz point}, which we recalled after Theorem~\ref{thm:niiro-sawashima}, can also be defined -- in the same way -- for closed rather than bounded operators. Hence, condition~\ref{thm:uniform-convergence-finite-residuum:item:simple-pole} in Theorem~\ref{thm:uniform-convergence-finite-residuum} can be rephrased as follows: $\spb(A)$ is a Riesz point of $A$ and its order as a pole of the resolvent is $1$.

Whereas the proof of Theorem~\ref{thm:uniform-convergence} made use of Lemma~\ref{lem:nci-dominant-bounded}, for the proof of Theorem~\ref{thm:uniform-convergence-finite-residuum}, we give crux of the argument in the following lemma. It provides a second criterion for the peripheral spectrum of $A$ to consist solely of the spectral bound of $A$, when $A$ is the generator of an eventually positive semigroup that is norm continuous at infinity. A proof for the particular case when the semigroup is positive can be found in \cite[Proposition~3.2]{MartinezMazon1996}. In fact, in case the semigroup is positive, not only is the pole order of $\spb(A)$ irrelevant but there is no requirement for the spectral space to be finite-dimensional either.  Since the cyclicity result that underlies \cite[Proposition~3.2]{MartinezMazon1996} is (currently) not available in the eventual positivity case, so we base our argument on the Niiro--Sawashima type result from Theorem~\ref{thm:niiro-sawashima} instead.

\begin{lemma}
	\label{lem:nci-dominant-pole}
	
	On a complex Banach lattice $E$, let $(e^{tA})_{t \in [0,\infty)}$ be a uniformly eventually positive semigroup that is norm continuous at infinity. If $\spb(A)$ is a first order pole of the resolvent $\Res(\argument,A)$ and the corresponding spectral space is finite-dimensional, then $\perSpec(A) = \left\{ \spb(A) \right\}$.
\end{lemma}

\begin{proof}
	Without loss of generality, we assume that the spectral bound of $A$ is $0$. Due to norm continuity at infinity, the growth bound is equal to the spectral bound, as shown in \cite[Corollary~1.4]{MartinezMazon1996}, and so the growth bound of $A$ is also $0$. As in the proof of Lemma~\ref{lem:nci-dominant-bounded}, there exists a number $\alpha>0$ such that $\perSpec(A) \subseteq [-i\alpha,i\alpha]$, and
	\begin{equation*}
		\spec(e^{tA}) \cap \bbT \subseteq \{e^{it\gamma}:\gamma \in [-\alpha,\alpha]\}
	\end{equation*}
for all $t \geq 0$. In fact, even a spectral mapping theorem for the intersection of the peripheral spectrum and the essential spectrum holds (\cite[Theorem~1.3]{MartinezMazon1996}), and hence $1$ is a Riesz point of $e^{tA}$ for all sufficiently small times $t$ -- namely for all $t \in (0,t_0)$, where $t_0 := \frac{\pi}{\alpha}$. 
	From now on, let $t \in (0,t_0)$.
	
	We show next that the order of the number $1$ as a pole of the resolvent $\Res(\argument,e^{tA})$ is $1$. Let $P$ denote the spectral projection associated to the pole $0$ of the resolvent $\Res(\argument,A)$. Then $0$ is not a spectral value of $A\restrict{\ker P}$. Now there are two possibilities for for the behaviour of the restricted semigroup $(e^{tA}\restrict{\ker P})_{t \in [0,\infty)}$: Either its growth bound is negative -- in which case $1$ is not a spectral value of $e^{tA}\restrict{\ker P}$ -- or its growth bound equals $0$. In the latter case, since the restricted semigroup is also norm-continuous at infinity, we can apply the spectral mapping theorem for the peripheral spectrum \cite[Theorem~1.2]{MartinezMazon1996} and conclude, again, that $1$ is not a spectral value of $e^{tA}\restrict{\ker P}$ (since $t \in (0,t_0)$). Hence, we only have to consider the pole order of the resolvent of $e^{tA}\restrict{\Ima P}$ at $1$. But since $\Res(\argument,A)$ has a first order pole at $0$, the operator $A$ acts as the zero operator on $\Ima P$, and so the operator
	\begin{align*}
		e^{tA}\restrict{\Ima P} = e^{t\big(A\restrict{\Ima P}\big)}
	\end{align*}
	acts as the identity on $\Ima P$. Therefore $1$ is indeed a first order pole the resolvent of $\Res(\argument,e^{tA}\restrict{\Ima P})$, and thus also of $\Res(\argument, e^{tA})$.
	
	 We can now apply our Niiro--Sawashima type theorem (Theorem~\ref{thm:niiro-sawashima}), which yields that every spectral value of $e^{tA}$ on the unit circle is a Riesz point and a first order pole of the resolvent. Consequently, $e^{tA}$ is power bounded and thus we may proceed as in the proof of Lemma~\ref{lem:nci-dominant-bounded} to conclude $\perSpec(A) =\{0\}$.
\end{proof}

\begin{proof}[Proof of Theorem~\ref{thm:uniform-convergence-finite-residuum}]
	Suppose the conditions \ref{thm:uniform-convergence-finite-residuum:item:norm-continuous-at-infinity} and \ref{thm:uniform-convergence-finite-residuum:item:simple-pole} hold. Then by Lemma~\ref{lem:nci-dominant-pole}, the peripheral spectrum of $A$ consists of $\spb(A)$ only. Thus $(e^{tA})_{t \in [0,\infty)}$ fulfils all the conditions adequate for a semigroup to converge in the operator norm topology to a finite rank operator as $t \to \infty$; see \cite[Theorem~3.3]{Thieme1998}.
	
	In the same reference, it is also shown that conditions \ref{thm:uniform-convergence-finite-residuum:item:norm-continuous-at-infinity} and \ref{thm:uniform-convergence-finite-residuum:item:simple-pole} are necessary for operator norm convergence of the semigroup.
\end{proof}

\subsection*{Acknowledgements} 

The first named author was supported by Deutscher Aka\-de\-mi\-scher Aus\-tausch\-dienst (Forschung\-sstipendium-Promotion in Deutschland).

\appendix

\section{Ultrapowers and Riesz points of linear operators}
\label{section:ultrapowers-riesz-points}

In this appendix, we recall a few facts about ultrapowers of Banach spaces and Riesz points. For further details about ultrapowers, we refer the reader to \cite{Heinrich1980}, \cite[p.~251--253]{Meyer-Nieberg1991}, and \cite[Section~V.1]{Schaefer1974}. By a \emph{Riesz point} of a bounded linear operator, we mean a spectral value that is a pole of the resolvent and has finite-dimensional spectral space. Details about such points can, for instance, be found in \cite{Barnes1999} or \cite{Barnes2005}; for related results, we also refer to the classical references \cite[Section~III.6.5]{Kato1980}, \cite[pp.\,330--332]{TaylorLay1986} and \cite[Section~VIII.8]{Yosida1980}.

\subsection*{A reminder of ultrapowers}
\label{subsection:reminder-of-ultrapowers}

Let $E$ be a (real or complex) Banach space and denote by $l^{\infty}(E)$, the space of all $E$-valued bounded sequences endowed with the canonical supremum norm, i.e., ${\norm{x}_{\infty}=\sup_{n \in \bbN} \norm{x_n}}$ for ${x=(x_n)_{n \in \bbN} \in l^{\infty}(E)}$.

Fix a free ultrafilter $\calU$ on $\bbN$ and define
	\[
		c_{\calU}(E):= \{ (x_n) \in l^{\infty}(E) : \lim_{\calU} \norm{x_n} =0 \}.
	\]
Then $c_{\calU}(E)$ is a closed subspace of $l^{\infty}(E)$ and the $\calU$-product $E^{\calU}$ denotes the quotient space
	\[
		E^{\calU} : = l^{\infty}(E) / c_{\calU}(E),
	\]
and is called the \emph{ultrapower of $E$ with respect to the ultrafilter $\calU$} or briefly the \emph{$\calU$-ultrapower of $E$}. For every element $x$ of $l^{\infty}(E)$, the equivalence class of $x$ in $c_{\calU}(E)$ will be denoted by $x^{\calU}$. It turns out that for every $x = (x_n) \in l^{\infty}(E)$, the norm of $x^{\calU}$ in $E^{\calU}$ is given by
	\[
		\norm{x^U} = \lim_{\calU} \norm{x_n}.
	\]
A proof of this can be found, for instance in \cite[Proposition~V.1.2]{Schaefer1974} for complex Banach spaces; the proof for real Banach spaces, however, remains the same.

Note that the above ultrapower construction is only interesting for infinite-dimensional Banach spaces: If $E$ is finite-dimensional, then its unit ball is compact, so it follows that $E^{\calU}$ is isomorphic to $E$.

If $E$ and $F$ are Banach spaces, then every bounded linear operator ${T \in \calL(E,F)}$ can be extended to a bounded linear operator ${T^{\calU} \in \calL\left(E^{\calU},F^{\calU}\right)}$ in a canonical way, i.e., ${T^{\calU} x^{\calU} = (Tx_n)^{\calU}}$ for every ${x=(x_n) \in l^{\infty}(E)}$. If ${E=F}$, then the mapping ${T \mapsto T^{\calU}}$ is an isometric homomorphism from the Banach algebra $\calL(E)$ to the Banach algebra $\calL\left(E^{\calU}\right)$ which preserves the identity element \cite[Proposition~V.1.2]{Schaefer1974}.

Ultrapowers play a significant role in operator theory. A part of this is because not only does the lifting ${T \mapsto {T^\calU}}$ preserve several spectral properties of $T$ but it also improves some. As an example, we have the following proposition, which we quote from \cite[Theorem~V.1.4 and its Corollary]{Schaefer1974}.

\begin{proposition}
	\label{prop:ultrapower-spectral-properties}

	Let $E$ be a complex Banach space and $\calU$ be a free ultrafilter on $\bbN$. If $T$ is a bounded operator on $E$, then the lifting $T^{\calU}$ has the following properties:
	\begin{enumerate}
		\item The spectrum of $T^{\calU}$ is exactly the spectrum of $T$.
		
		\item Both $T$ and $T^{\calU}$ have the same approximate point spectrum $\appSpec$ which is the same as the point spectrum of $T^{\calU}$, i.e.,
			\[
				\pntSpec\left(T^{\calU}\right)=\appSpec\left(T^{\calU}\right)=\appSpec(T).
			\]
		\item The lifting of the resolvent of $T$ is equivalent to resolvent of the lifting $T^{\calU}$. In other words, ${\Res(\argument,T)^{\calU} = \Res(\argument,T^{\calU})}$.
	\end{enumerate}
	In particular, a spectral value of $T$ is a $k$-th order pole of the resolvent $\Res(\argument,T)$ if and only if it is a $k$-th order pole of the resolvent $\Res(\argument,T^{\calU})$.
\end{proposition}

We remark in passing that both operators $T$ and $T^{\calU}$ have the same operator norm and if $E$ is Banach lattice, then $T$ is positive if and only if $T^{\calU}$ is positive.

\subsection*{A characterization of Riesz points}

Let $E$ be a complex Banach space and $T$ be a bounded linear operator on $E$. If $T$ is a contraction, then it was shown by Groh in \cite[Proposition~3.2]{Groh1984} that a spectral value $\lambda$ of $T$ with modulus $\modulus{\lambda}=1$ is a pole of the resolvent $\Res(\argument,T)$, if the corresponding eigenspace of a lifted operator $T^{\calU}$ is finite-dimensional. Then Caselles showed in \cite[Proposition~3.3]{Caselles1987} that the same result is true under more general assumptions. Even more generally, we have the following result:

\begin{proposition}
	\label{prop:riesz-point-via-ultrapower}
	
	Let $T$ be a bounded linear operator on a complex Banach space $E$ and let $\calU$ be a free ultrafilter on $\bbN$. For any complex number $\lambda$, the following assertions are equivalent:
	\begin{enumerate}[ref=(\roman*)]
		\item\label{prop:riesz-point-via-ultrapower:item:riesz-point-of-T} $\lambda$ is a Riesz point of $T$.
		
		\item\label{prop:riesz-point-via-ultrapower:item:riesz-point-of-lifting} $\lambda$ is a Riesz point of $T^{\calU}$.
		
		\item\label{prop:riesz-point-via-ultrapower:item:eigenspace} $\lambda$ is an element of the topological boundary of $\spec(T)$ and the eigenspace $\ker\left(\lambda I-T^{\calU}\right)$ is finite-dimensional.
	\end{enumerate}
	If the equivalent assertions {\upshape(i)}--{\upshape(iii)} are satisfied and $P$ denotes the spectral projection of $T$ associated with $\lambda$, then $P^\calU$ is the spectral projection of $T^\calU$ associated with $\lambda$.
\end{proposition}
	
\begin{proof}
	The implication \ref{prop:riesz-point-via-ultrapower:item:riesz-point-of-lifting} $\Rightarrow$ \ref{prop:riesz-point-via-ultrapower:item:eigenspace} is obvious and the implication \ref{prop:riesz-point-via-ultrapower:item:eigenspace} $\Rightarrow$ \ref{prop:riesz-point-via-ultrapower:item:riesz-point-of-T} was shown in \cite[Corollary~3.2]{Glueck2020}.

	Let us now prove the implication \ref{prop:riesz-point-via-ultrapower:item:riesz-point-of-T} $\Rightarrow$ \ref{prop:riesz-point-via-ultrapower:item:riesz-point-of-lifting} 
	and the assertion at the end of the proposition. If $\lambda$ is a Riesz point of $T$, then by Proposition~\ref{prop:ultrapower-spectral-properties}, $\lambda$ is a pole of the resolvent $\Res\left(\argument,T^{\calU}\right)$. To see that the corresponding spectral space is finite-dimensional, we argue as in the proof of \cite[Proposition~3.3]{Caselles1987}: If $P$ is the spectral projection of $T$ associated to $\lambda$, then
	\[
		P = \int_C \Res(\mu,T)\, d\mu
	\]
	(where $C$ is a circle with centre $\lambda$ such that it encloses no other point of the spectrum of $T$). Using continuity of the map $T \mapsto T^{\calU}$ and ${\Res(\argument,T)^{\calU} = \Res(\argument,T^{\calU})}$, we obtain that $P^{\calU}$ is the spectral projection of $T^{\calU}$ corresponding to the spectral value $\lambda$. As $\Ima P$ is finite-dimensional, it follows that $\Ima P^{\calU} = (\Ima P)^{\calU}$ is finite-dimensional, as well.
\end{proof}

We note here that, in order to check by means of Proposition~\ref{prop:riesz-point-via-ultrapower} that $\lambda$ is a Riesz point of $T$, it suffices to check the corresponding eigenspace of $T^{\calU}$ is finite-dimensional for some, rather than each, free ultrafilter $\calU$ on $\bbN$.

\bibliographystyle{plain}
\bibliography{literature}

\begin{thebibliography}{10}

\bibitem{AddonaGregorioRhandiTacelli2021}
Davide {Addona}, Federica {Gregorio}, Abdelaziz {Rhandi}, and Cristian
  {Tacelli}.
\newblock {Bi-Kolmogorov type operators and weighted Rellich's inequalities},
  2021.
\newblock Preprint. Available online at https://arxiv.org/abs/2104.03811v1.

\bibitem{Arora2021}
Sahiba {Arora}.
\newblock {Locally eventually positive operator semigroups}.
\newblock {\em {To appear in J. Oper. Theory}}.
\newblock Preprint available online at https://arxiv.org/abs/2101.11386v2.

\bibitem{AroraGlueck2021}
Sahiba {Arora} and Jochen {Gl\"uck}.
\newblock {Uniform (anti-)maximum principles and eventual positivity}.
\newblock Preprint. Available online at https://arxiv.org/abs/2104.12205v1.

\bibitem{Barnes1999}
Bruce~A. {Barnes}.
\newblock {Riesz points and Weyl's theorem}.
\newblock {\em {Integral Equations Oper. Theory}}, 34(2):187--196, June 1999.

\bibitem{Barnes2005}
Bruce~A. {Barnes}.
\newblock Riesz points of upper triangular operator matrices.
\newblock {\em Proceedings of the American Mathematical Society},
  133(5):1343--1347, 2005.

\bibitem{BatkaiKramarRhandi2017}
Andr\'as {B\'atkai}, Marjeta {Kramar Fijav\v{z}}, and Abdelaziz {Rhandi}.
\newblock {\em {Positive operator semigroups: From finite to infinite
  dimensions}}, volume 257.
\newblock Basel: Birkh\"auser/Springer, 2017.

\bibitem{Blake1999}
Mark~D. {Blake}.
\newblock {\em Asymptotically norm-continuous semigroups of operators}.
\newblock PhD thesis, University of Oxford, 1999.

\bibitem{Caselles1987}
Vicent {Caselles}.
\newblock {On the peripheral spectrum of positive operators}.
\newblock {\em {Isr. J. Math.}}, 58:144--160, 1987.

\bibitem{Daners2014}
Daniel {Daners}.
\newblock {Non-positivity of the semigroup generated by the
  Dirichlet-to-Neumann operator}.
\newblock {\em {Positivity}}, 18(2):235--256, 2014.

\bibitem{DanersGlueck2017}
Daniel {Daners} and Jochen {Gl\"uck}.
\newblock {The role of domination and smoothing conditions in the theory of
  eventually positive semigroups}.
\newblock {\em {Bull. Aust. Math. Soc.}}, 96(2):286--298, 2017.

\bibitem{DanersGlueck2018b}
Daniel {Daners} and Jochen {Gl\"uck}.
\newblock {A criterion for the uniform eventual positivity of operator
  semigroups}.
\newblock {\em {Integral Equations Oper. Theory}}, 90(4):19, 2018.
\newblock Id/No 46.

\bibitem{DanersGlueck2018a}
Daniel {Daners} and Jochen {Gl\"uck}.
\newblock {Towards a perturbation theory for eventually positive semigroups}.
\newblock {\em {J. Oper. Theory}}, 79(2):345--372, 2018.

\bibitem{DanersGlueckKennedy2016b}
Daniel {Daners}, Jochen {Gl\"uck}, and James~B. {Kennedy}.
\newblock {Eventually and asymptotically positive semigroups on Banach
  lattices}.
\newblock {\em {J. Differ. Equations}}, 261(5):2607--2649, 2016.

\bibitem{DanersGlueckKennedy2016a}
Daniel {Daners}, Jochen {Gl\"uck}, and James~B. {Kennedy}.
\newblock {Eventually positive semigroups of linear operators}.
\newblock {\em {J. Math. Anal. Appl.}}, 433(2):1561--1593, 2016.

\bibitem{DenkKunzePloss2020}
Robert {Denk}, Markus {Kunze}, and David {Plo{\ss}}.
\newblock {The Bi-Laplacian with Wentzell boundary conditions on Lipschitz
  domains}.
\newblock {\em {To appear in Integral Equations Oper. Theory}}.

\bibitem{EngelNagel2000}
Klaus-Jochen {Engel} and Rainer {Nagel}.
\newblock {\em {One-parameter semigroups for linear evolution equations}},
  volume 194.
\newblock Berlin: Springer, 2000.

\bibitem{FerreroGazzolaGrunau2008}
Alberto {Ferrero}, Filippo {Gazzola}, and Hans-Christoph {Grunau}.
\newblock {Decay and local eventual positivity for biharmonic parabolic
  equations}.
\newblock {\em {Discrete Contin. Dyn. Syst.}}, 21(4):1129--1157, 2008.

\bibitem{GazzolaGrunau2008}
Filippo {Gazzola} and Hans-Christoph {Grunau}.
\newblock {Eventual local positivity for a biharmonic heat equation in
  \(\mathbb R^n\)}.
\newblock {\em {Discrete Contin. Dyn. Syst., Ser. S}}, 1(1):83--87, 2008.

\bibitem{Glueck2016}
Jochen {Gl\"uck}.
\newblock {\em {Invariant sets and long time behaviour of operator
  semigroups}}.
\newblock PhD thesis, Universit{\"a}t Ulm, 2016.
\newblock DOI: 10.18725/OPARU-4238.

\bibitem{Glueck2017}
Jochen {Gl\"uck}.
\newblock {Towards a Perron-Frobenius theory for eventually positive
  operators}.
\newblock {\em {J. Math. Anal. Appl.}}, 453(1):317--337, 2017.

\bibitem{Glueck2020}
Jochen {Gl\"uck}.
\newblock {Spectral gaps for hyperbounded operators}.
\newblock {\em {Adv. Math.}}, 362:24, 2020.
\newblock Id/No 106958.

\bibitem{GlueckMugnolo2021}
Jochen {Gl\"uck} and Delio {Mugnolo}.
\newblock {Eventual domination for linear evolution equations}.
\newblock {\em {to appear in Math. Z.}}

\bibitem{GregorioMugnolo2020}
Federica {Gregorio} and Delio {Mugnolo}.
\newblock {Bi-Laplacians on graphs and networks}.
\newblock {\em Journal of Evolution Equations}, 20(1):191--232, March 2020.

\bibitem{Groh1984}
Ulrich {Groh}.
\newblock {Uniformly ergodic maps on \(C^*\)-algebras}.
\newblock {\em {Isr. J. Math.}}, 47:227--235, 1984.

\bibitem{Heinrich1980}
Stefan {Heinrich}.
\newblock {Ultraproducts in Banach space theory}.
\newblock {\em Journal für die reine und angewandte Mathematik}, 313:72--104,
  1980.

\bibitem{Kato1980}
Tosio {Kato}.
\newblock {\em {Perturbation theory for linear operators. Corr. printing of the
  2nd ed}}, volume 132.
\newblock Springer, Berlin, 1980.

\bibitem{Lotz1986}
Heinrich~P. {Lotz}.
\newblock {Positive linear operators on \(L^p\) and the Doeblin condition}.
\newblock In {\em {Aspects of positivity in functional analysis. Proceedings of
  the Conference held on the Occasion of H. H. Schaefer's 60th Birthday,
  T\"ubingen, June 24-28, 1985}}, pages 137--156. 1986.

\bibitem{LotzSchaefer1968}
Heinrich~P. {Lotz} and Helmut~H. {Schaefer}.
\newblock {\"Uber einen Satz von F. Niiro und I. Sawashima}.
\newblock {\em {Math. Z.}}, 108:33--36, 1968.

\bibitem{LuxemburgZaanen1971}
Wilhelmus A.~J. {Luxemburg} and Adriaan~C. {Zaanen}.
\newblock {\em {Riesz spaces}}, volume~1.
\newblock Elsevier (North-Holland), Amsterdam, 1971.

\bibitem{MagalRuan2018}
Pierre {Magal} and Shigui {Ruan}.
\newblock {\em {Theory and applications of abstract semilinear Cauchy
  problems}}.
\newblock Springer International Publishing, Cham, 2018.

\bibitem{Martinez1993}
Josep {Mart\'{\i}nez}.
\newblock {The essential spectral radius of dominated positive operators}.
\newblock {\em {Proc. Am. Math. Soc.}}, 118(2):419--426, 1993.

\bibitem{MartinezMazon1996}
Josep {Mart\'{\i}nez} and Jos\'e~M. {Maz\'on}.
\newblock {\(C_0\)-semigroups norm continuous at infinity}.
\newblock {\em {Semigroup Forum}}, 52(2):213--224, 1996.

\bibitem{Meyer-Nieberg1991}
Peter {Meyer-Nieberg}.
\newblock {\em {Banach lattices}}.
\newblock Berlin etc.: Springer-Verlag, 1991.

\bibitem{Nagel1986}
Rainer {Nagel}, editor.
\newblock {\em {One-parameter semigroups of positive operators}}, volume 1184.
\newblock Springer, Cham, 1986.

\bibitem{NagelPoland2000}
Rainer {Nagel} and Jan {Poland}.
\newblock {The critical spectrum of a strongly continuous semigroup}.
\newblock {\em Advances in Mathematics}, 152(1):120--133, 2000.

\bibitem{NiiroSawashima1966a}
Fumio {Niiro} and Ikuko {Sawashima}.
\newblock {On positive irreducible operators in an arbitrary Banach lattice and
  a problem of H. H. Schaefer}.
\newblock {\em {Proc. Japan Acad.}}, 42:677--681, 1966.

\bibitem{NiiroSawashima1966b}
Fumio {Niiro} and Ikuko {Sawashima}.
\newblock {On the spectral properties of positive irreducible operators in an
  arbitrary Banach lattice and problems of H. H. Schaefer}.
\newblock {\em {Sci. Pap. Coll. Gen. Educ., Univ. Tokyo}}, 16:145--183, 1966.

\bibitem{NoutsosTsatsomeros2008}
Dimitrios {Noutsos} and Michael~J. {Tsatsomeros}.
\newblock {Reachability and holdability of nonnegative states}.
\newblock {\em {SIAM J. Matrix Anal. Appl.}}, 30(2):700--712, 2008.

\bibitem{Schaefer1974}
Helmut~H. {Schaefer}.
\newblock {\em {Banach lattices and positive operators}}, volume 215.
\newblock Springer, Berlin, 1974.

\bibitem{TaylorLay1986}
Angus~E. {Taylor} and David~C. {Lay}.
\newblock {\em {Introduction to functional analysis}}.
\newblock {Kreiger, 1986}, reprint of the second edition, 1986.

\bibitem{Thieme1998}
Horst~R. {Thieme}.
\newblock {Balanced exponential growth of operator semigroups}.
\newblock {\em {J. Math. Anal. Appl.}}, 223(1):30--49, 1998.

\bibitem{Webb1987}
Glenn~F. {Webb}.
\newblock {An operator-theoretic formulation of asynchronous exponential
  growth}.
\newblock {\em Transactions of the American Mathematical Society},
  303(2):751--763, 1987.

\bibitem{Yosida1980}
Kosaku {Yosida}.
\newblock {\em {Functional analysis. 6th ed}}, volume 123.
\newblock Springer, Berlin, 1980.

\bibitem{Zaanen1983}
Adriaan~C. {Zaanen}.
\newblock {\em {Riesz spaces II}}, volume~30.
\newblock Elsevier (North-Holland), Amsterdam, 1983.

\end{thebibliography}

\end{document}